\documentclass[a4paper,11pt,reqno]{amsart}
\usepackage[T1]{fontenc}
\usepackage[utf8]{inputenc}
\usepackage{lmodern}
\usepackage{amsmath, amsthm, amssymb,amscd, mathrsfs, amsfonts, mathtools}
\usepackage{hyperref}
\usepackage{euler}
\usepackage{times}
\usepackage[all]{xy}
\usepackage{todonotes}
\usepackage{xcolor}
\usepackage[lite]{amsrefs}

\renewcommand{\PrintDOI}[1]{\href{http://dx.doi.org/\detokenize{#1}}{doi: \detokenize{#1}}%
  \IfEmptyBibField{pages}{, (to appear in print)}{}}

\def\commutatif{\ar@{}[rd]|{\circlearrowleft}}

\newcommand{\eq}[1][r]
   {\ar@<-3pt>@{-}[#1]
    \ar@<-1pt>@{}[#1]|<{}="gauche"
    \ar@<+0pt>@{}[#1]|-{}="milieu"
    \ar@<+1pt>@{}[#1]|>{}="droite"
    \ar@/^2pt/@{-}"gauche";"milieu"
    \ar@/_2pt/@{-}"milieu";"droite"}

\def\dar[#1]{\ar@<2pt>[#1]\ar@<-2pt>[#1]}
  \entrymodifiers={!!<0pt,0.7ex>+} %%%%%%%%%%%%%%%%%

\newcommand{\bigon}[4][r]{% %%%%% bigons
    \ar@/^1pc/[#1]^{#2}_*=<0.3pt>{}="HAUT"
    \ar@/_1pc/[#1]_{#3}^*=<0.3pt>{}="BAS"
    \ar@{=>} "HAUT";"BAS" ^{#4}
  }

\newcommand{\bigons}[6][r]{  %%%%% Vertical composition of bigons
    \ar@/^2pc/[#1]^{#2}_*=<0.3pt>{}="HAUT"
    \ar@{}    [#1]     ^*=<0.3pt>{}="MILIEUHAUT"
                       _*=<0.3pt>{}="MILIEUBAS"
    \ar[#1]_(0.3){#3}                  
    \ar@/_2pc/[#1]_{#4}^*=<0.3pt>{}="BAS"
    \ar@{=>} "HAUT";"MILIEUHAUT" ^{#5}
    \ar@{=>} "MILIEUBAS";"BAS" ^{#6}
  }

\newtheorem{thm}{Theorem}[section]
\newtheorem{pro}[thm]{Proposition}
\newtheorem{lem}[thm]{Lemma}
\newtheorem{cor}[thm]{Corollary}

\theoremstyle{definition}
\newtheorem{df}[thm]{Definition}
\newtheorem{dfpro}[thm]{Definition and Proposition}
\newtheorem{dflem}[thm]{Definition and Lemma}

\theoremstyle{remark}
\newtheorem{rmk}[thm]{Remark}

\newtheorem{ex}[thm]{Example}
\newtheorem{exs}[thm]{Examples}
\allowdisplaybreaks

\newcommand\id{1}
\newcommand{\Id}{\operatorname{id}}

\newcommand{\R}{\mathbb{R}}

\newcommand{\Z}{\mathbb{Z}}

\newcommand\rTo{\longrightarrow}

\newcommand\mto{\longmapsto}
\newcommand\Rinjection{\hookrightarrow}
\newcommand\rRack{\triangleleft}

\newcommand\Rack{\diamond}

\let\cal\mathcal
\let\cat\mathfrak

\let\bb\mathbb

\def\a{\alpha}

\def\s{\sigma}
\def\vp{\varphi}

\hypersetup{
  colorlinks = true,
  urlcolor = blue,
  linkcolor = blue,
  citecolor = red,
  pdfauthor = {Elhamdadi, M., Moutuou, E.},
  pdfkeywords = {Quandles, extensions of quandles},
  pdftitle = {Elhamdadi, M., Moutuou, E. - Extensions of topological quandles},
  pdfsubject = {Higher groupoids},
  pdfpagemode = UseNone
}

\title{Foundations of topological racks and quandles}

\author{Mohamed Elhamdadi} 
\address{Department of Mathematics, 
University of South Florida, Tampa, FL 33620, U.S.A.} 
\email{emohamed@math.usf.edu} 

\author{El-ka\"ioum M. Moutuou} 
\address{School of Mathematics, 
University of Southampton
Highfield
Southampton
SO17 1BJ
UK} 
\email{e.mohamedmoutuou@soton.ac.uk}

\dedicatory{Dedicated to Professor J\'ozef H. Przytycki for his 60th birthday}

\begin{document}

\maketitle

\begin{abstract}
We give a foundational account on topological racks and quandles. Specifically, we define the notions of ideals, kernels, units, and inner automorphism group in the context of topological racks. Further, we investigate topological rack modules and principal rack bundles. Central extensions of topological racks are then introduced providing a first step towards a general continuous cohomology theory for topological racks and quandles. 
\end{abstract}

\tableofcontents

\section{Introduction}
Quandles are non-associative algebraic structures (with the exception of the trivial quandles) that correspond to the axiomatization of the three Reidemeister moves in knot theory. Since 1982 when quandles were introduced by Joyce~\cite{Joyce} and Matveev~\cite{Matveev} independently, there have been investigations, (see for example \cite{PrzyYang, PrzSik, NP, Przy, EN, Nelson, Nosaka1}), that have mostly focused on {\em finite} quandles. Joyce and Matveev proved that the {\it fundamental} quandle of a knot is a complete invariant up to orientation.  Precisely, given two knots $K_0$ and $K_1$, the fundamental quandle $Q(K_0)$ is isomorphic to the fundamental quandle $Q(K_1)$  if and only if $K_1$ is equivalent to $K_0$ or $K_1$ is equivalent to the reverse of the mirror image of $K_0$.  Quandles have been used by topologists to construct invariants of knots in the 3-space and knotted surfaces in 4-space.  We mention the following two examples of invariants: (1) the set of colorings of a given knot by a quandle (see \cite{CESY} for example), (2) state sum invariants of knots and knotted surfaces coming from quandle cohomology \cite{CES, CJKLS, Nosaka}. Topological quandles were considered in 2007 by Rubinsztein~\cites{Rubinsztein:Top_Quandles}.  He extended the notion of coloring of a knot or link by a quandle to include topological quandles.  He showed that the coloring space of the link is a topological space (defined up to a homeomorphism). Jacobsson and Rubinsztein \cite{Jacobsson-Rubinsztein} computed the space of colorings of all prime knots with up to seven crossings and of all $(2,n)$-torus links. They also observed some similarities between the space of colorings of knots and Khovanov homology for all prime knots with up to seven crossings and for at least some eight-crossing knots.    \\

In this paper, we introduce the foundational material to investigate topological racks and quandles. In section 2, we review the basics of topological racks and then introduce the notion of units in a topological racks. These form a space that can be thought of as a generalisation of the center of a topological group (cf. Proposition~\ref{pro:units_center}). The inner automorphism group of a topological quandle is constructed in section 3 and its topology is discussed. 
In section 4 we introduce the notions of ideal and kernel for topological racks and we give some first foundamental results. 
We go further by exploring in Section 5 modules and rack group bundles over topological racks which are crucial to the study of central extensions of topological racks we define in Section 6. We then form an abelian group out of such extensions that outlines a general continuous cohomology theory for topological racks. This will appear in a subsequent paper~\cite{Elhamdadi-Moutuou:Continuous_Cohomology}.

\section{Racks and quandles}

Recall~\cites{Rubinsztein:Top_Quandles, Etingof-Grana:Rack_cohomology, Carter-Crans-Elhamdadi-Saito} that a 
%(right)
 {\it rack} is a set $X$ provided with a binary operation 
\[
\begin{array}{lccc}
\rRack: & X\times X & \rTo &X \\
 & (x,y) & \mto & x\rRack y
\end{array}
\] such that 
\begin{itemize}
\item[(i)] for all $x,y\in X$, there is a unique $z\in X$ such that $y=z\rRack x$;
\item[(ii)]({\it right distributivity}) for all $x,y,z\in X$, we have $(x\rRack y)\rRack z=(x\rRack z)\rRack (y\rRack z)$.
\end{itemize}

Observe that property (i) also reads that for any fixed element $x\in X$, the map $R_x:X\ni y\mto y\rRack x\in X$ is a %surjection
bijection. %In fact, $R_x$ is also injective, thanks to property (ii). 
Also, notice that the distributivity condition is equivalent to the relation $R_x(y\rRack z)=R_x(y)\rRack R_x(z)$ for all $y,z\in X$.

A {\it topological rack} is a 
%(right)
 rack $X$ which is a topological space such that the map $X\times X\ni (x,y)\mto x\rRack y\in X$ is a continuous. In a topological rack, the right multiplication $R_x:X\ni y\mto y\rRack x\in X$ is a homeomorphism, for all $x\in X$. Observe that an ordinary (finite) rack is automatically a topological rack with respect to the discrete topology.

%Unless otherwise stated, we will always assume our racks (or quandles) to be right racks (or quandles).

\begin{df}
A {\it quandle} (resp. {\it topological quandle}) is a rack (resp. topological rack) such that 
 $x\rRack x=x,\forall x\in X$. 
\end{df}

\begin{rmk}
Suppose that a set (resp. a topological space) $X$ is equipped with a binary operation $\Rack:X\times X\ni (x,y)\mto x\Rack y \in X$ 
%with respect to which $X$ is a right and left rack (resp. topological rack) 
that is right and left distributive at the same time. Then $(X,\Rack)$ is a quandle (resp. topological quandle). Indeed, for all $x\in X$, we have 
\[
R_{x\Rack x}(x)=x\Rack (x\Rack x)=(x\Rack x)\Rack (x\Rack x)=R_{x\Rack x}(x\Rack x),
\] 
which implies that $x\Rack x=x$.
\end{rmk}

\begin{ex}[The conjugation quandle]
 Let $G$ be a topological group. The operation $$x\rRack y=yxy^{-1}$$ makes $G$ into a topological quandle which is denoted by $Conj(G)$ and is called the {\em conjugation quandle of $G$}. In fact, any conjugacy class of $G$ is a topological quandle with this operation.
\end{ex}

\begin{ex}[The core quandle]
 Let $G$ be a topological group. The operation $$x\rRack y=yx^{-1}y$$ defines a topological quandle structure on $G$. This quandle will be denoted by $Core(G)$ and we call it the {\em core of $G$}. Observe that this operation satisfies $(x\rRack y)\rRack y=x$. Any quandle in which this equation is satisfied is called an {\it involutive} quandle.
 \end{ex}

\begin{ex}[Symmetric manifold]
First recall that a symmetric manifold $M$ is a Riemannian manifold such that each point $x \in M$ is an isolated fixed point of an involtutive isometry $i_x:M \rightarrow M$.   Given such manifold, every $x\in M$ endows $M$ with the structure of topological quandle by setting $x \rRack  y=i_y(x)$.
\end{ex}

\begin{ex}
 Let $S^n$ be the unit sphere of $\bb R^{n+1}$. Then, with respect to the operation  $$x \rRack y=2(x \cdot y)y-x, x,y\in S^n,$$ where $x\cdot y$ is the usual scalar product in $\bb R^{n+1}$, and the topology inherited from $\bb R^{n+1}$, $S^n$ is a topological quandle.
\end{ex}

\begin{ex}
Following the previous example, let $\lambda$ and $\mu$ be real numbers, and let $x,y \in S^n$. Then $$\lambda x \rRack \mu y=\lambda[2{\mu}^2(x\cdot y)y-x].$$ 
In particular, the operation $$\pm x \rRack  \pm y=\pm (x \rRack y)$$ provides a structure of topological quandle on the projective space $\mathbb{RP}^n$.
\end{ex}

\begin{ex}
Let $G$ be a topological group and $\sigma$ be a homeomorphism of $G$.  Let $H$ be a closed subgroup of $G$ such that $\sigma(h)=h,$ for all $ h \in H$.  Then $G/H$ is a quandle with operation $$[x]\rRack [y]:=[\s(x)\s(y)^{-1}y],$$ where for $x\in G$, $[x]$ denotes the class of $x$ in $G/H$. For example, one can consider the group $G$ to be the group of rotations $G=SO(2n+1)$,  $H=SO(2n)$ and $G/H=S^{2n+1}$.
\end{ex}

\begin{df}\label{stu}
Let $X$ be a topological rack or quandle. An element $u\in X$ is
\begin{enumerate}
\item a {\it stabiliser} if $x\rRack u=x$, for all $x\in X$;
\item {\it totally fixed} in $X$ if $u \rRack x=u$, for all $x\in X$;
\item a {\it unit} if $u$ is a stabiliser and is totally fixed in $X$.
\end{enumerate}
The set of all stabilisers of $X$ (resp. all totally fixed points in $X$) is denoted by $Stab(X)$ (resp. $Fix(X)$)
\end{df}

Observe that if $u$ is a stabiliser in the rack $X$, we have $u \rRack u=u$. Moreover, if $u$ is a unit, then $(x\rRack u)\rRack y=x\rRack y$ for all $x,y\in X$.

\begin{lem}
Assume the topological rack $X$ admits a non-empty set of units. Then for all arbitrary pair of units $u,v$ we have 
\[
u\rRack v=u, \ {\rm \ } v\rRack u=v.
\]
\end{lem}

\begin{proof}
Indeed, if $u$ and $v$ are units in $X$, then by (1) and (2) in the definition \ref{stu}, we have $u\rRack v=u$ and $u\rRack v=u$.
\end{proof}

\begin{df}
The set of all units in a topological racks or quandle $X$ is denoted by $\cal U_X$. We say that $X$ is {\it unital} if $\cal U_X$ is non-empty.
\end{df}

\begin{ex}\label{ex:units_vs_centre}
Let $G$ be a topological group. Then it is easy to check that $\cal U_{Conj(G)}$ is exactly the centre $Z(G)$ of $G$.
\end{ex}

\begin{ex}[Topological Linear rack]\label{TLR}
Let $G$ be a topological group and $V$ a continuous representation; {\em i.e.}, there is a continuous map $$G\times V\ni (g,v)\mto  g\cdot u\in V$$ with $g\cdot (h\cdot v)=(gh)\cdot v, \;$ for all $g,h\in G, v\in V$. We define a topological structure on $G\times V$ as follows:
\[
(g,u)\rRack (h,v):= (h^{-1}gh, h^{-1}\cdot u), \ g,h\in G, u,v\in V.
\]
We denote this rack as $G\ltimes V$. Observe that this rack is unital and $(1,0)$ is a unit.
\end{ex} 

The following proposition is immediate.

\begin{pro}
Let $G$ be a topological group and $V$ a countinuous representation through the map $\pi:G\rTo GL(V)$. Denote by $V^G$ the subspace of $V$ consisting of invariant vectors under the continuous $G$-action. Then we have 
\[
Stab(G\ltimes V) \cong [Z(G)\cap \ker(\pi)] \times V, \ Fix(G\ltimes V)\cong Z(G)\times V^G, 
\]
and 
\[
\cal U_{G\ltimes V}\cong [Z(G)\cap \ker(\pi)]\times V^G.
\]
\end{pro}

\begin{df}
Let $X$ and $Y$ be topological racks. A {\it rack morphism} from $X$ to $Y$ is a continuous map $f:X\rTo Y$ such that $f(x\rRack y)=f(x)\rRack f(y)$, for all $x,y\in X$. Morphisms of topological quandles are defined in the same way. Isomorphisms of racks or quandles are defined accordingly. If $Y$ is unital, then $f$ is said to be {\it unital} if $f(\cal U_X)\subseteq \cal U_Y$.
\end{df}

\begin{ex}
Given a topological rack $X$, each element $x\in X$ defines a rack automorphism through $R_x:X\ni y\mto y\rRack x\in X$. Moreover, if $X$ is unital, $R_x$ is a unital morphism.
\end{ex}

\begin{pro}\label{pro:units_center}
Let $G$ be a topological group. Then every unit element in $Core(G)$ is a $2$-torsion of the group $G$. In particular, if $G$ is torsion free, $\cal U_{Core(G)}$ is empty. 
\end{pro}

For instance, $Core(\bb R)$ has no units.

\begin{ex}
The classical map $f: \R \rightarrow S^1$ given by $f(t)=e^{2i\pi t}$ is a quandle homomorphism from $\R$ with the binary operation $t \rRack t'=2t'-t$ to the quandle $S^1$ with operation  $z \rRack z'=z' z^{-1} z'$.
\end{ex}

\begin{dflem}\label{dflem:unitarization} %{dflem:unitisation}
Let $X$ be a non-unital topological rack. Define the %{\it unitisation} 
{\it unitarization} $X^+$ of $X$ by adding a one point set $\{\id\}$ to $X$ and declaring that $x\rRack \id=x$ and $\id\rRack x=\id$ for all $x\in X$ and endowing it with the topology induced from the inclusion map $X\ni x\mto x\rRack \id \in X^+$. Then $X^+$ is a unital topological rack .
 Moreover, the inclusion $X\Rinjection X^+$ is an injective morphism of topological racks.
\end{dflem}

\begin{proof}
Straightforward.
\end{proof}

\begin{rmk}
Notice that $u\in X$ is a stabiliser if and only if $R_u$ is the identity morphism of racks $X\rTo X$. Further, $u$ is totally fixed if and only if it is a fixed point of $R_x$ for every $x\in X$. It follows that in the Definition and Lemma~\ref{dflem:unitarization}, we have changed nothing in the "structure" of $X$ since the added unit $1$ may be identified with the identity morphism of the racks $\Id: X\rTo X$ 
and be considered as a fixed point of all of the morphisms $R_x$.
\end{rmk}

%%%%%%
\section{Inner automorphism group}

Let $X$ be a topological rack. Notice that if $f,g:X\rTo X$ are (continuous) rack morphisms then so is $fg:=f\circ g$. If moreover $f$ and $g$ are rack automorphisms ({\it i.e.}, rack homeomorphisms), then so is $fg$. The set $Aut(X)$ of rack automorphisms forms a group under composition. Furthermore, when equipped with the compact-open topology, $Aut(X)$ is a topological group. Recall that the right translation $R_x:X\rTo X$ is an automorphism of topological rack.

\begin{pro}\label{pro:inner_representation}
Define the \emph{inner representation} of  $X$ to be the map 
\[
\begin{array}{lccc}
R: & X & \rTo &Aut(X)\\
 & x & \mto & R_x
\end{array}
\]
Then $R$ is continuous. Moreover, for all $x,x'\in X$, we have 
\[
R_xR_{x'}(\cdot)=R_x(\cdot)\rRack R_x(x'). %ME added (\cdot) to the left side of the equation
\]
\end{pro}

We shall note that the compact-open topology has basis open sets 
\[
W(K,U):=\left\{f:X\rTo X \ {\rm rack \ homomorphism} \mid f(K)\subset U \right\},
\]
where $K\subset X$ is compact and $U\subset X$ is open. We then need the following lemma to prove the proposition \ref{pro:inner_representation}.

\begin{lem}\label{lem:inner_representation}
Let $X$ be a topological rack and let $K$ and $U$ be compact and open subsets of $X$, respectively. Suppose there exists $x\in X$ such that $K\rRack x\subset U$. Then there is an open neighbourhood $V$ of $x$ such that $K\rRack V\subset U$.
\end{lem}

\begin{proof}
Since the rack operation $X\times X\ni (y,x)\mto y\rRack x\in X$ is a continuous map and $U$ open, there exit open neighbourhoods $\widetilde{V}_{x,y}$ and $V_{x,y}$ of $y$ and $x$, respectively, such that $$\widetilde{V}_{x,y}\rRack V_{x,y}\subset U.$$
Now, for a fixed $x\in X$, the family $\{\widetilde{V}_{x,y}\}_{y\in K}$ is an open cover of the compact subset $K\subset X$. Hence, there is a finite set $\{y_0,\cdots, y_n\}\subset K$ such that $$K\subset \bigcup_{k=0}^n\widetilde{V}_{x,y_k}, \ {\rm and\ } \widetilde{V}_{x,y_k}\rRack V_{x,y_k}\subset U.$$
It is straightforward that the open neighbourhood $$V_x:=\bigcap_{k=0}^nV_{x,y_k}$$ of $x$ satisfies $K\rRack V_x\subset U$. 
\end{proof}

\begin{proof}[Proof of Proposition~\ref{pro:inner_representation}]
Let $W(K,U)$ be an open subset in $Aut(X)$. Then thanks to Lemma~\ref{lem:inner_representation}, if $x$ is in the inverse image of $W(K,U)$ by $R$, there is an open neighbourhood $V_x$ such that $V_x\subset R^{-1}(W(K,U))$; hence, $R^{-1}(W(K,U))$ is open in $X$ and $R$ is then continuous. For the second statement, we have $$R_xR_{x'}(y)=(y\rRack x')\rRack x=(y\rRack x)\rRack (x'\rRack x)=R_x(y)\rRack R_x(x'), \forall y\in X.$$
\end{proof}

%\begin{dfpro}
\begin{df}
We define the {\it inner automorphism group} $Inn(X)$ of $X$ to be the closure of the subgroup generated by the image of $X$ by $R$ in $Aut(X)$; {\it i.e.}, 
\[
Inn(X):= \overline{<R(X)>}\subset Aut(X).
\]
%$Inn(X)$ is a closed subgroup of $Aut(X)$. Morover, 
\end{df}
%\end{dfpro}

Recall that for any quandle endomomorphism $f$ of $X$, we have $f \;R_x=R_{f(x)}\; f$.  Then  $Inn(X)$ is a normal subgroup of $Aut(X)$ as the closure of a normal subgroup.  With the quotient topology, $Aut(X)/Inn(X)$ is a topological group.  Also, since $R$ is continuous, if $X$ is compact, then $Inn(X)$ is a compactly generated group.

\begin{ex}
Consider again the core of $\R$. Then $Aut(Core(\bb R))$ is the affine group $Aff(\R)=\{ 
\begin{pmatrix}
 a & b \\
  0 & 1
 \end{pmatrix}, 0 \neq a, b \in \R  \}$ and the inner group $Inn(Core(\R))=\R$.
\end{ex}

\begin{ex} 
Let $M (\neq I_2)$ be an invertible two-by-two matrix over the integers $\Z$ (i.e. $det(M)=\pm 1$), where $I_2$ is the identity matrix, and assume that $M^2 \neq I_2$.  The plane $\R^2$ becomes a topological quandle with the operation $x \rRack y=Mx +(I_2-M)y$.  It is easily seen that this map is compatible with the projection of $\R^2 \rightarrow \R^2/\Z^2$.  Let $m$ and $n$ be two vectors of $\Z^2$.  We have $(x + m) \rRack (y+n)= x \rRack y + m \rRack n $.  Since $m \rRack n  \in \Z^2$, we obtain a quandle operation on the torus $T^2=S^1 \times S^1$.
Lets compute the automorphism group $Aut(T^2)$.  First, one notices that any function $f_{A,B}$ on $\R^2$ such that $f_{A,B}(x)=Ax+B$ with the condition $MA=AM$ is a quandle homomorphism.   Thus if $A \in GL_2(\R)$ and $MA=AM$, then $f_{A,B}$ is an automorphism of the quandle $\R^2$.  In fact we claim that the converse is also true.  Precisely if $f$ is a quandle automorphism and we consider the function $g(x)=f(x)-f(0)$.  Then $g(0)=0$ and $g$ satisfies the equation $g(Mx +(I_2-M)y)=Mg(x) +(I_2-M)g(y)$.  In particular $g(Mx)=Mg(x)$, and thus $g$ will be of the form $g(x)=\lambda x$, where $\lambda \in GL_2(\R)$ and $\lambda M=M \lambda$.   Thus $Aut(T^2)$ is the subgroup of the affine group $Aff(\R^2)$ of elements of the form $f_{A,B}$ for which $A$ commute with $M$ and the inner group $Inn(T^2)=\R^2$.  Obviously this example can be generalised to an $n$-torus with $n\geq 2$.
\end{ex}

\section{Ideals and Kernels}

In this section, we generalise the notion of {\it ideals} to the category of topological racks and quandles.

\begin{df}
Let $X$ be a topological racks (resp. quandle). A {\it subrack} (resp. {\it subquandle}) of $X$ is a topological subspace $Y\subset X$ such that $x\rRack y\in Y$ whenever $x,y\in Y$. A subrack or subquandle is {\it closed} (resp. {\it open} if it is closed (resp. open) as a subspace of a topological space.
\end{df}

Notice that a subrack (resp. subquandle) $Y$ of $X$ is a rack (resp. quandle) in its own. Moreover, we have the following straightforward observation.

\begin{dflem}
Let $Y$ be a subrack of $X$. Let 
\[X\rRack Y:=\{x\rRack y, \ x\in X, y\in Y\}, \ Y\rRack X:=\{y\rRack x, \ y\in Y, x\in X\}.
\]
Then the operation $$(y_1\rRack x_1)\star (y_2\rRack x_2):= (y_1\rRack y_2)\rRack (x_1\rRack x_2),$$ for $(y_i,x_i)\in Y\times X, i=1,2$, provides  $Y\rRack X$ with the structure of a (right) topological rack. Note that the topology of $Y\rRack X$ is induced from that of $X$.
\end{dflem}

\begin{df}
A right (resp. left) {\it ideal} of a topological rack (resp. quandle)  $X$ is a closed subrack (resp. subquandle) $Y$ of $X$ such that $X\rRack Y\subseteq Y$ (resp. $Y\rRack X\subseteq Y$).
If $Y$ is a right and left ideal of $X$ at the same time, we will say that $Y$ is a two-sided ideal, or simply an ideal of $X$.
\end{df}

\begin{ex}
Let $G$ be a topological group endowed with the usual topological quandle conjugation structure $x\rRack y=y^{-1}xy$. Then, if $N$ is a closed normal subgroup of $G$, we have $n\rRack g=g^{-1}ng\in N$, for all $g\in G, n\in N$; hence, $N$ is a left ideal of the quandle $Conj(G)$. Conversely, it is straightforward for the definition of the quandle structure of $G$ that if $N$ is a left ideal of the topological quandle $G$, then $N$ is closed subgroup of $G$.
\end{ex}

\begin{df}
A left (resp. right) ideal in a topological rack or quandle $X$ is called {\it proper} if it is not empty and is not (homeomorphic) to the whole $X$. 
\end{df}

\begin{pro}
Assume $X$ is a topological rack with units. Then $X$ admits no proper right ideal.
\end{pro}

\begin{proof}
First, note that if $I$ is a non-empty right ideal in $X$, then $Fix(X)\subset I$; for if $u$ is totally fixed, then for all $y\in I$, we have $u=u\rRack y\in I$. Now, if $u\in Fix(X)\cap Stab(X)$, then $u\in I$, and we have $x=x\rRack u\in I$ for all $x\in X$. In other words, $X=I$. 
\end{proof}

\begin{df}
Let $f:X\rTo Y$ be a morphism of topological racks. We define the {\it kernel} of $f$ as
\[
\ker f := \{x\in X\mid f(x)\in \cal U_Y\}.
\]
\end{df}

We immediately have the following observation.

\begin{pro}
Let $f:X\rTo Y$ be a morphism of topological racks or quandles. Then $\ker f$ is a left closed ideal in $X$.
\end{pro}

\begin{proof}
Let $x\in \ker f$ and $x'\in X$. Then, since $f(x)$ is totally fixed in $Y$, we have for all $y\in Y$ 
\[
y\rRack f(x\rRack x')=y\rRack(f(x)\rRack f(x'))=y\rRack f(x)=y,
\]
which implies that $f(x\rRack x')\in Stab(Y)$; and
\[
f(x\rRack x')\rRack y = (f(x)\rRack f(x'))\rRack y=f(x)\rRack y=f(x)=f(x\rRack x'),
\]
which implies that $f(x\rRack x')\in Fix(Y)$. Hence, $x\rRack x'\in \ker f$, and $\ker f\rRack X \subseteq \ker f$.
\end{proof}

We justify the terminology "kernel" of rack morphisms by the following lemma.

\begin{pro}
Let $X$ and $Y$ be topological racks with $Y$ unital, and let $f:X\rTo Y$ be a unital morphism. If $f$ is injective, then  $\ker f =\cal U_X$.
\end{pro}

\begin{proof}
Suppose $f$ injective and let $x_0\in \ker f$. Then, for all $x\in X$, we have 
\[
f(x_0\rRack x)= f(x_0)\rRack f(x)=f(x_0), 
\]
which implies $x_0\rRack x=x_0$; {\it i.e.}, $x_0\in Fix(X)$. Further, 
\[
f(x\rRack x_0)=f(x)\rRack f(x_0)=f(x),
\]
so that $x\rRack x_0=x$; {\it i.e.}, $x_0\in Stab(X)$. We then have shown that $\ker f\subseteq Fix(X)\cap Stab(X)=\cal U_X$. 
\end{proof}

\begin{rmk}
Note that the converse of the above lemma is not true in general. Indeed, let $G$ and $\Gamma$ be topological groups with trivial centres. Any group homomoprhism $f:G\rTo \Gamma$ induces a quandle homomorphism $Q_f:G\rTo \Gamma$ where $G$ and $\Gamma$ are given the usual quandle structure $x\rRack y:=y^{-1}xy$. Moreover, it is easy to check that $f$ is an injective group homomorphism if and only if $Q_f$ is an injective quandle homomorphism. Now, thanks to Example~\ref{ex:units_vs_centre} we see that $\cal U_G$ and $\cal U_{\Gamma}$ are trivial and we obviously have $\ker Q_f=\cal U_G=\{e\}$ for all group homomorphism $f$.
\end{rmk}

%%%%%%%
\section{Topological rack modules}

In this section we define and study modules over topological racks. 

\noindent Let $X$ be a topological space. By a {\em group bundle} over $X$ we mean a topological space $\cal A$ together with a surjective open continuous map $\pi:\cal A\rTo X$ such that each fibre $\cal A_x, x\in X$, ({\em i.e.} the pre-image $\pi^{-1}(x)\subset \cal A$) is a topological group.

\begin{df}
Let $X$ be topological rack. A {\em rack group bundle} over $X$ consists of a pair $(\cal A, \eta)$ where $\cal A$ is a group bundle over $X$ and $\eta$ is a family of isomorphisms $\eta_{x,y}:\cal A_x\rTo \cal A_{x\rRack y}$ such that 
\[
\eta_{x \rRack y,z}\; \eta_{x,y} =  \eta_{x \rRack z,y \rRack z}\; \eta_{x,z}
\]
for all $x,y,z\in X$.
\end{df}

\begin{df}\label{module}
Let $X$ be a topological rack. An $X$-module is a triple $\cat A=(\cal A,\eta,\tau)$ where $(\cal A,\eta)$ is a rack group bundle over $X$ and $\tau$ is a family of topological group morphisms $\tau_{x,y}:\cal A_y\rTo \cal A_{x\rRack y}$ such that

\begin{enumerate}

\item $\cal A_x$ is abelian for all $x\in X$;
\item $ \eta_{x \rRack y,z}\; \tau_{x,y} =  \tau_{x \rRack z,y \rRack z}\; \eta_{y,z}$; and
\item $\tau_{x \rRack y,z}= \eta_{x \rRack z,y \rRack z} \tau_{x,z}+
 \tau_{x \rRack z,y \rRack z}\tau_{y,z}$.
\end{enumerate}
Moreover, if $X$ is a quandle, we require the following axiom
\begin{enumerate}
  \item[(4)] $\tau_{x,x} + \eta_{x,x} =id_{\cal A_x}$.
\end{enumerate}
\end{df}

Observe that our definition coincides with the definition of~\cites{And-Grana, Carter-Elhamdadi-Grana-Saito}  when $X$ is given the discrete topology and when $A_x$ is a fixed abelian group $A$ for all $x \in X$. 
The first two identities of the definition \ref{module} can be understood as the following two commutative diagrams,% as in~\cite{Jackson:Extensions_racks}, 
\[
\xymatrix{
A_x\ar[r]^{\eta_{x,y}} \ar[d]_{\eta_{x,z}} & A_{x \rRack y } \ar[d]^{\eta_{x \rRack y,z}}
\\ A_{x \rRack z } \ar[r]_{\eta_{x \rRack z,y \rRack z}} & A_{(x \rRack y)\rRack z }
}
\]
 
and

\[
\xymatrix{
A_y\ar[r]^{\tau_{x,y}} \ar[d]_{\eta_{y,z}} & A_{x \rRack y } \ar[d]^{\eta_{x \rRack y,z}}
\\ A_{y \rRack z }  \ar[r]_{\tau_{x \rRack z,y \rRack z}} & A_{(x \rRack y)\rRack z }
}
\]

All the following examples correspond to the case when $A_x$ is a fixed abelian group $A$ for all $x \in X$.
\begin{exs}
\begin{enumerate}
\item
Let $X$ be a topological rack and $A$ be a topological abelian group. Take $\eta_{x,y}$ to be the identity map and $\tau_{x,y}$ to be the zero map. Then $A$ is trivially a topological $X$--module.
\item
Let $\Lambda = \bb Z[t,t^{-1}]$ denote the ring of Laurent polynomials. 
Then any  $\Lambda$-module $A$ is an $X$-module for any quandle $X$,
by $\eta_{x,y} (a)=ta$ and $\tau_{x,y} (b) = (1-t) (b) $
for any $x,y \in X$.
\item
Given a topological rack $X$ (we may need to assume that $X$ is completely regular space),  recall that the free topological group $F(X)$ on $X$ is defined to be the unique (up to topological isomorphism) topological group such that (1) the injection $i: X \rTo F(x)$ is continous, and (2) for any topological group $G$ and a continuous map $\phi: X \rightarrow $, there is a unique continuous homomorphism $\Phi: F(x) \rightarrow G$, such that $\phi=\Phi \circ i$.
 
 let $G_X$ be the topological quotient  group  $F(x)/N$, where $N$ is the normal subgroup generated by $\langle \ x\rRack y -yxy^{-1} \rangle$. Any $G_X$-module $A$  is a  $X$-module by 
 $\eta_{x,y} (a)=y a$ and $\tau_{x,y} (b) = b- (x \rRack y)b$,
where $x, y \in X$, $a, b \in A$.
\end{enumerate}
\end{exs}

\begin{pro}[Rack semidirect product]
Let $X$ be a topological rack and $\cat A=(\cal A,\eta,\tau)$ be an $X$-- module. Let the set 
\[
\cat A\ltimes X:= \left\{(a,x)\in \cal A\times X \mid a\in \cal A_x\right\}
\]
be equipped with the topology induced from that of the product topology of $\cal A\times X$. Then, under the operation 
\begin{eqnarray}\label{eq:semidirect}
(a,x)\rRack (b,y):=(\eta_{x,y}(a)+\tau_{x,y}(b),x\rRack y),
\end{eqnarray}
$\cat A\ltimes X$ is a topological rack called the {\em rack  semidirect product} of $\cat A$ and $X$.
\end{pro}

\begin{proof}
We omit the algebraic verifications since they are similar as in the proof of~\cite[Proposition 2.1]{Jackson:Extensions_racks}. It remains to check that the operation~\eqref{eq:semidirect} is continuous when $\cat A\ltimes X$ is endowed with the induced topology from $\cal A\times X$. Let then $\cal O\times U$ be an open subset of $\cat A\ltimes X$ and $((a,x),(b,y))$ be in the pre-image $F$ of $\cal O\times U$ through the binary operation~\eqref{eq:semidirect} so that we have
\[
(\eta_{x,y}(a)+\tau_{x,y}(b),x\rRack y)\in \cal O\times U\subset \cal A\times X.
\] 
In particular $x\rRack y\in U$ and since the rack operation of $X$ is continuous, there exist open sets $V,W\subset X$ such that $x\in V$ and $y\in W$. Further, since the group operation in $\cal A_{x\rRack y}$ is continuous and $\eta_{x,y}(a)+\tau_{x, y}(b)\in \cal O\subset \cal A_{x\rRack y}$, there exist two open subsets $C'$ and $D'$ in $\cal A_{x\rRack y}$ containing $\eta_{x,y}(a)$ and $\tau_{x,y}(b)$, respectively. Now, by continuity of the morphisms $\eta_{x,y}$ and $\tau_{x,y}$, we can find open subsets $C$ and $D$ of $\cal A_x$ and $\cal A_y$ containing $a$ and $b$, respectively. It follows that $F\subset (C\times U)\times (D\times V)$ is open $(\cat A\ltimes X)\times (\cat A\ltimes X)$.
\end{proof}

\begin{ex}\label{TLR2}
Let $V$ be a continuous representation of a topological group $G$ as in example~\ref{TLR}. Then the first projection $G\times V \ni (g,u)\mto g\in Conj(G)$ defines a rack group bundle over the conjugation rack $Conj(G)$ by setting 
\[
\eta_{g,h}(v):=h^{-1}\cdot v, \ g,h \in G, v\in V.
\]
Furthermore, it is straightforward to check that $(G\times V,\eta, 0)$ is a topological rack $Conj(G)$--module where $0$ is the zero map on the vector space $V$.
\end{ex}

\noindent
The following is a generalization of example \ref{TLR2}.

\begin{ex}
Let $V$ be a continuous representation of a topological group $G$ and $\a:G \times G \rightarrow V$ be a mapping. Consider the binary operation on $V\times G$ given as follows:
\[
(u,g)\rRack (v,h):= (h^{-1}\cdot u + \a(g,h), h^{-1}gh), \ g,h\in G, u,v\in V.
\]
Then it is easly seen that this binary operation gives a rack structure on $V\times G$ if and only if the map $\a$ is a {\em cocycle}; that is, $\alpha$ satisfies the following condition, for all $g,h,k \in G$
$$k^{-1}\a(g,h)+\a(h^{-1}gh,k)=k^{-1}h^{-1}k\a(g,k)+\a(k^{-1}gk,k^{-1}hk).
$$ 
%Observe that this rack is unital and $(1,0)$ is a unit.
In this case the the topological rack thus obtained if denote by $V\rtimes_\a G$. Next, it is straightforward to see that the projection $pr_2:V\rtimes_\a G\ni (u,g)\mto g\in Conj(G)$ is a rack group bundle with fibre the abelian group $V$ with $\eta_{g,h}(v):= h^{-1}\cdot v$, for $g,h\in G, u\in V$. Moreover, by setting $\tau_{g,h}(v)=0\in V$ for all $g,h\in G, v\in V$, we turn $V\rtimes_rG$ into a $Conj(G)$--module. 
\end{ex}

\section{Extensions of topological racks}

 In this section we define the notion of central extensions of topological racks by rack modules. We recall from~\cite[Definition 2.2]{Rubinsztein:Top_Quandles} that given a topological rack $X$ and a topological space $M$, a {\em continuous rack action} of $X$ on $M$ consists of a continuous map $$M\times X\ni (m,x)\mto m\cdot x \in M$$ such that 
 \[
 (m\cdot x)\cdot y=(m\cdot y)\cdot (x\rRack y), \ \forall m\in M, x,y\in X.
 \]

\begin{ex}
Let $X$ be a topological rack and denote by $\underline{X}$ its underlying topological space. Then the binary operation $\rRack: \underline{X}\times X\rTo \underline{X}$ defines a continuous rack $X$--action on $\underline{X}$.
\end{ex}

\begin{ex}
Suppose $X$ is a topological rack and $\cat A$ is an $X$-module. Then the topological space $\cat A\ltimes X$ is naturally equipped with continuous rack action of $X$ as follows:
\[
(a,x)\cdot y := (\eta_{x,y}(a),x\rRack y), (a,x)\in \cat A\ltimes X, y\in X.
\]
\end{ex}

\begin{rmk}
We shall observe that any continuous right %left 
action $M\times A\rTo M$ of a topological group $A$ on a topological space $M$ is actually a rack action of the topological quandle $Conj(A)$ on $M$ (cf.~\cite[Example 2.9]{Rubinsztein:Top_Quandles}). Whence, in the sequel we will not distinguish between continuous action of a topological $A$ in the usual sense and the induced rack action of its conjugation rack.
\end{rmk}

\begin{df} 
Let $A$ be a topological group. Suppose $E$ is a topological rack with a continuous $A$-action. Let $p:E\rTo X$ be a surjective rack homomorphism with local continuous sections. We say that $(E,p)$ is an $A$-principal rack bundle if the fibres $E_x:=p^{-1}(x)$ are transitive with respect to the $A$-action; {\em i.e.}, for all $e,e'\in E_x$ there is a unique $a(e,e')\in A$ such that $e'=e\cdot a(e,e')$.
\end{df}

We immediately have the following observation. 

\begin{lem}
If $p:E\rTo X$ is an $A$-- principal rack bundle, then for all local section $s:U\rTo E$ of $p$ ({\em i.e.}, $p\circ s=\Id_U$) we get a homeomorphism $$E_{|U}\stackrel{\cong}{\rTo} U\times A$$ as follows: for $e\in E_{|U}$, let $z=p(e)\in U$, then since $E$ is $A$-principal and $s(z),e\in E_z$, there exists a unique $a(s(z),e)\in A$ such that $e=s(z)\cdot a(s(z),e)$. We then define $E_{|U}\ni e\mto \left(p(e),a(s(p(e)),e)\right)\in U\times A$. And $$U\times A\ni (x,a)\mto x\cdot a\in E_{|U}.$$
\end{lem} 
 
\begin{df}
Let $X$ be a topological rack and $\cat A=(\cal A,\eta,\tau)$ be an $X$--module. A {\em central $\cat A$--extension} of $X$ consists of
\begin{itemize}
\item a topological rack $E$;
\item a surjective rack homomorphism $p:E\rTo X$ with continuous local sections;
\item a continuous $\cal A$--principal action of $E$; that is a continuous map 
\[
E\times_{X}\cal A \ni (e,a)\mto e\cdot a\in E, 
\]
where $E\times_X\cal A=\{(e,a)\in E\times \cal A\mid p(e)=\pi(a)\in X\}$, such that for all $x\in X$ and $e,e'\in E_x$, there is a unique element $a(e,e')\in \cal A_x$ with $e'=e\cdot a(e,e')$,
\end{itemize}
satisfying the following axioms
\begin{itemize}
\item[(1)] for all $(e,a)\in E\times_X\cal A$ and all $f\in E$ with $p(f)=y\in X$, we have 
\[
(e\cdot a)\rRack f=(e\rRack f)\cdot \eta_{x,y}(a); 
\] 
\item[(2)] for all $e\in E$ with $p(e)=x\in X$ and all $(f,b)\in E\times_X\cal A$, we have 
\[
e\rRack (f\cdot b)=(e\rRack f)\cdot \tau_{x,y}(b).
\]
\end{itemize}
Such an central $\cat A$--extension is represented as $(E,p)$.
\end{df} 

\begin{pro}[Trivial extension]
Let $\cat A$ be an $X$--module. Then the $(\cat A\ltimes X,\tilde{\pi})$ is a central $\cal A$-extension, where the projection $\tilde{\pi}:\cat A\ltimes X\rTo X$ is given by $\tilde{\pi}(a,x)=x=\pi(a)$ and the $\cal A_x$--action on $\cal A_x$ is by multiplication on the topological abelian group $\cal A_x$.
\end{pro}

The proof is straightforward, so we omit it.

\begin{df}
Let $(E,p)$ and $(F,q)$ be two central $\cat A$-extensions of $X$. 
\begin{enumerate}
\item A morphism $\vp:(E,p)\rTo (F,q)$ is a topological rack homomorphism $\vp:E\rTo F$ which is a bundle morphism and $\cal A$--equivariant in the sense that the following diagrams commute
\[
\xymatrix{
E\ar[rr]^\vp \ar[rd]_p && F \ar[ld]^q \\ 
& X &
}
\]
and
\[
\xymatrix{
E\times_X\cal A \ar[r] \ar[d]_{\vp\times Id} & E \ar[d]^\vp \\
F\times_X\cal A \ar[r] & F
}
\]
where in the horizontal arrows in the second diagram are the $\cal A$--actions; {\em i.e.}, $$q(\vp(e))=p(e), \forall e\in E, \ {\rm and\ } \vp(e)\cdot a=\vp(e\cdot a), \forall (e,a)\in E\times_X\cal A.$$
\item We say that $(E,p)$ and $(F,q)$ are {\em equivalent}, and we write $(E,p)\sim (F,q)$, if there exists a morphism $\vp:(E,p)\rTo (F,q)$ which is an isomorphism of topological racks whose inverse $\vp^{-1}:F\rTo E$ is also a morphism of central $\cat A$--extensions. In this case, we say that $\vp$ is an {\em equivalence} of central $\cat A$--extensions. We denote by $Ext(X,\cat A)$ the set of equivalence classes of central $\cat A$--extensions of $X$.
\item The extension $(E,p)$ is said to be {\em trivial} if it is equivalent to the trivial central $\cat A$--extension $(\cat A\ltimes X,\tilde{\pi})$.
\end{enumerate}
\end{df}

\begin{dfpro}[Baer sum]
Let $(E,p)$ and $(F,q)$ be central $\cat A$--extensions of $X$. Consider the equivalence relation "$\sim$" in
\[
E\times_XF:= \{(e,f)\in E\times F\mid p(e)=q(f)\}
\]
given by $(e\cdot a,f)\sim (e,f\cdot a)$ for $(e,a)\in E\times_X\cat A$, and define the topological space $E\sqcup_XF$ to be the quotient space. We denote by $[e,f]$ the class of $(e,f)\times E\times_XF$ in $E\sqcup_XF$.  
Then, with respect to the binary operation 
\[
[e_1,f_1]\rRack [e_2,f_2]:=[e_1\rRack e_2,f_1\rRack f_2],
\]
$E\sqcup_XF$ is a topological rack. Furthermore, $E\sqcup_XF$ is equipped with the continuous $\cat A$--principal action 
\[
[e,f]\cdot a:=[e\cdot a,f]=[e,f\cdot a], \ (e,a)\in E\times_X\cal A,
\]
and the projection $p:E\sqcup_XF \ni [e,f] \mto p(e)=q(f)\in X$ makes $(E\sqcup_XF,p)$ into a central $\cat A$--extension of $X$ which we call the {\em Baer sum} of $(E,p)$ and $(F,q)$.
\end{dfpro}

\begin{pro}
Let $(E,p)$ be a representative of a class in $Ext(X,\cat A)$. Let $(E^\circ,p^\circ)$ be the central $\cat A$--extension of $X$ where $E^\circ$ is $E$ as a topological space, $p^\circ:E^\circ\rTo X$ is the projection of $E$ ({\em i.e.}, for $e^\circ\in E^\circ, \ p^\circ(e^\circ):=p(e)$, where we write $e^\circ$ for $e\in E$ viewed as an element in $E^\circ$), and the continuous $\cal A$--principal action is given by 
\[
e^\circ \cdot a:=(e\cdot a^{-1})^\circ, \ (e^\circ,a)\in E^\circ\times_X \cal A.
\]
Then the central $\cat A$--extension $(E\sqcup_XE^\circ,p)$ is trivial. We call $(E^\circ,p^\circ)$ the {\em opposite} of $(E,p)$.
\end{pro}

\begin{proof}
Define the map $\psi:E\sqcup_XE^\circ \rTo \cat A\ltimes X$ by
\[
\psi([e,f^\circ]):=(a(e,f),p(e)),
\]
where $a(e,f)$ is the unique element in $\cal A_{p(e)}=\cal A_{p(f)}$ such that $f=e\cdot a(e,f)$. To see that $\psi$ is well defined, take $(e,b)\in E\times_X\cal A$, and $(e,f^\circ)\times E\times_XE^\circ$. Then, we have $e=f\cdot a(e,f)^{-1}$, so that
\[
\begin{array}{ll}
f\cdot b &  =  (e\cdot b)\cdot a(e\cdot b,f\cdot b) \\ 
 &=  (f\cdot a(e,f)^{-1}b)\cdot a(e\cdot b,f\cdot b)\\
 &=  (f\cdot b)\cdot (a(e,f)^{-1}a(e\cdot b,f\cdot b))
\end{array}
\]
since $\cal A_{p(e)}$ is an abelian group. Therefore, since $E$ is $\cal A$--principal, the element $a(e,f)^{-1}a(e\cdot b,f\cdot b)$ is unique and must then be equal to the identity in $\cal A_{p(e)}$. In other words, $a(e,f)=a(e\cdot b,f\cdot b)$, and $\psi([e,f])=\psi([e\cdot b,f^\circ\cdot b^{-1}])$. It is a matter of easy check to see that $\psi$ is a morphism of central $\cat A$--extensions of $X$. 
Now, we get an inverse $\phi$ of $\psi$ by setting for all $(a,x)\in \cat A\ltimes X$, 
\[
\phi(a,x):= [e_x,(e_x\cdot a)^\circ],
\]
where $e_x$ is any element in the fibre $E_x$.
\end{proof}

\begin{cor}
Let $\cat A$ be an $X$--module. Then $Ext(X,\cat A)$ is an abelian group under Baer sum and inverse given by the equivalence class of the opposite extension. The zero element is the class of the trivial extension.
\end{cor}

A general theory of continuous cohomology of topological quandles is being developed by the authors in~\cite{Elhamdadi-Moutuou:Continuous_Cohomology}.

%\begin{bibdiv}\begin{biblist}
%\bibselect{references}
% \end{biblist}\end{bibdiv}

\end{document}